\documentclass[final,3p,times]{elsarticle}

\usepackage{amssymb}
\usepackage{graphics,graphicx,amsmath,latexsym,amssymb,amsthm,amssymb,amsfonts}
\newtheorem{thm}{Theorem}[section]

\newtheorem{defn}[thm]{Definition}

\newtheorem{exm}[thm]{Example}

\usepackage{enumerate}
\journal{XXX}
\begin{document}
\begin{frontmatter}
\author{J. Mahanta 
}
 \ead{jm$\_$nerist@yahoo.in}
 \author{P.K. Das 
}
\ead{pkd$\_$ma@yahoo.com}

\address{Department of Mathematics\\
North Eastern Regional Institute of Science and Technology\\
Nirjuli, Arunachal Pradesh, 791 109, INDIA.\\ }

\title{ON FUZZY BOUNDARY}

\begin{abstract}
In this paper, we introduce a new type fuzzy boundary and study some related set theoretic identities. Further, this new type of fuzzy boundary is compared with different existing fuzzy boundaries.
\end{abstract}

\begin{keyword}
Fuzzy boundary \sep Fuzzy topology. 
\MSC 06D72.
\end{keyword}
\end{frontmatter}

\section{Introduction}

From general topology, we know crisp boundary of a set is the set of elements which are both in the closure of the set and closure of its complement. But in many real life situations, crisp boundary of a situation may not be well defined. For instance, boundary of ocean shared by two or more countries and boundary of the states in India to be recognized from satellite pictures etc. Such boundaries may be termed as fuzzy boundaries. Further, the importance of fuzzy boundary is found in generalization of modern control theory as discussed in \cite{4} and \cite{7}.\\

Fuzzy boundary in the context of fuzzy topological space was defined by Warren \cite{WW}, which was later modified by Cuchillo-Ibanez and Tarres in \cite{CI}. Subsequently, Pu and Lie \cite{PI} have defined fuzzy boundary as a generalization of the crisp boundary in general topology. Below are some definitions and results as discussed in \cite{CI}, \cite{PI} and \cite{WW}.

\begin{defn} \cite{WW}
Let $A$ be a fuzzy set in an FTS $(X, \tau)$. The fuzzy boundary of $A$ is defined as the infimum of all closed fuzzy sets $D$ in $X$ with the property that $D(x) \geq cl(A(x))$ for all $x \in X$ for which $(cl(A) \bigwedge cl(A^c))(x) > 0$. We shall denote such boundary by $bd_IA$.
\end{defn}

Warren verified the following properties of the fuzzy boundary:

\begin{itemize}
\item [(1)] The boundary is closed.
\item [(2)] The closure is the supremum of the interior and boundary.
\item [(3)] The boundary reduces to the usual topological boundary when all fuzzy sets are crisp. 
\item [(4)] The boundary operator is an equivalent way of defining a fuzzy topology.
\end{itemize}

However, Warren's definition lacks the following properties:

\begin{itemize}
\item [(5)] The boundary of a fuzzy set is identical to the boundary of the complement of the set.
\item [(6)] If a fuzzy set is closed (or open), then the interior of the boundary is empty.
\item [(7)] If a fuzzy set both open and closed then the boundary is empty.
\end{itemize}

\begin{defn} \cite{PI}
Fuzzy boundary of a fuzzy set $A$ is $bd_{II}A=cl(A) \bigwedge cl(A^c)$. 
\end{defn}

\begin{defn} \cite{CI}
Let $A$ be a fuzzy set in an FTS $(X, \tau)$. The fuzzy boundary of $A$ is defined as the infimum of all closed fuzzy sets $D$ in $X$ with the property that $D(x) \geq cl(A(x))$ for all $x \in X$ for which $(cl(A) - int(A))(x) > 0$. We shall denote it by $bd_{III}A$.
\end{defn}

Cuchillo-Ibanez and Tarres established that their definition of fuzzy boundary satisfies the properties (1)-(4) and (7).

Later, Athar and Ahmad (\cite{Ahmad}, \cite{Athar}) introduced and studied fuzzy semiboundary by generalizing fuzzy boundary  through fuzzy semi-closed sets.

\section{Fuzzy boundary}

In this section, we introduce a new type of fuzzy boundary using the concepts of fuzzy closure and fuzzy interior. Our aim behind introducing the new definition is to move close to exactness.

\begin{defn}
Let $A$ be a fuzzy set in an FTS $(X, \tau)$. The boundary of $A$, denoted by $bdA$ is a fuzzy set defined as $bdA = int cl A \bigwedge cl intA$. 
\end{defn}

From the definition, followings can easily be concluded:
\begin{itemize}
\item $bd0_X= 0_X$ and $bd1_X= 1_X$.
\item The boundary of a set is same as the boundary of its complement.
\item If a fuzzy set is clopen then the boundary is equal to itself.
\item Boundary is not necessarily a closed fuzzy set.
\item If the closures and interiors of two fuzzy sets are equal then their boundaries are also equal.
\end{itemize}

\begin{exm}
Let $X = \{a, b\}$ and $\tau =\{0_X, \{a_{0.8}, b_{0.4}\}, \{a_{0.3}, b_{0.2}\}, \{a_{0.3}, b_{0.4}\}, \{a_{0.2}, b_{0.2}\}, 1_X\}$ be a fuzzy topology on $X$. Then boundary of the fuzzy set $\{a_{0.4}, b_{0.3}\}$ is $\{a_{0.3}, b_{0.4}\}$, which is not fuzzy closed.
\end{exm}

\begin{thm}
For any two fuzzy sets $A$ and $B$ in an FTS $(X, \tau)$, the following results hold.
\begin{enumerate}[(i)]
\item $bd(A \bigvee B) \geq bdA \bigvee bdB$;
\item $bd(A \bigwedge B)\geq bdA \bigwedge bdB$;
\item $intA \bigvee bdA = bdA$;
\item $bdA \leq clA$;
\item If $A$ is clopen then $bd(bdA)=A$;
\item $bd(clA) \geq bdA$;
\item $bd(intA) \leq bdA$;
\item If $A \leq B$ then $bdA < bdB$;
\item $A \bigvee bdA \leq clA$;
\item If $intA= intclA$ then $bd(clA) = bdA$;
\item If $clA= clintA$ then $bd(intA) = bdA$;
\item $A$ is closed $\Leftrightarrow~ bdA \leq A$.
\end{enumerate} 
\end{thm}

\begin{proof}
Let us consider fuzzy sets $A$ and $B$ in an FTS $(X, \tau)$. Then 
\begin{enumerate}[(i)]
\item \begin{eqnarray*}
bd(A \bigvee B) & = & intcl(A \bigvee B) \bigwedge clint(A \bigvee B) \\
&\geq& [int(clA \bigvee clB)] \bigwedge [cl(intA \bigvee int B)] \\
&\geq&  [intclA \bigvee intclB] \bigwedge [clintA \bigvee clint B]\\ 
&=& [intclA \bigwedge clintA] \bigvee [intclB \bigwedge clint B] \bigvee [intclA \bigwedge clint B]\\
&& \bigvee [intclB \bigwedge clintA]\\ 
&\geq& bdA \bigvee bdB. 
\end{eqnarray*}
\item \begin{eqnarray*}
bd(A \bigwedge B) & = & [intcl(A \bigwedge B)] \bigwedge [clint(A \bigwedge B)] \\
&\leq& [int(clA \bigwedge clB)] \bigwedge [cl(intA \bigwedge int B)] \\
&\leq&  [intclA \bigwedge intclB] \bigwedge [clintA \bigwedge clint B]\\ 
&=& [intclA \bigwedge clintA] \bigwedge [intclB \bigwedge clint B] \\ 
&=& bdA \bigwedge bdB. 
\end{eqnarray*}
\item \begin{eqnarray*}
intA \bigvee bdA &=& intA \bigvee [intclA \bigvee clintA]\\
&=& intclA \bigvee clintA ~~~~~~~~~~~~~~~~~~~~~~~~~~~~~ (since ~~intA \leq clintA)\\
&=& bdA.
\end{eqnarray*}
\item We know, $intclA \leq clA$ and $intA \leq A \Rightarrow clintA \leq clA$\\
i.e., $intclA \bigwedge clintA \leq clA$.
\item As $A$ is clopen, $bdA = A \Rightarrow bd(bdA) = bdA = A$.
\item \begin{eqnarray*}
bd(clA) &=& intcl(clA) \bigvee clint(clA)\\
&=& intclA \bigwedge clintclA\\
&\geq& intclA  \bigwedge clintA\\
&=& bdA.
\end{eqnarray*}
\item \begin{eqnarray*}
bd(intA) &=& intcl(intA) \bigvee clint(intA)\\
&=& intclintA \bigwedge clintA\\
&\leq& intclA  \bigwedge clintA\\
&=& bdA.
\end{eqnarray*}
\item $A \leq B \Rightarrow clA \leq clB$ and $intA \leq intB$.\\
So, $intclA \leq intclB$ and $clintA \leq clintB$.\\
$\Rightarrow bdA \leq bdB$.
\item Follows from $bdA \leq clA$ and $A \leq bdA$.
\item Straightforward.
\item Straightforward.
\item $A$ is closed $\Rightarrow clA = A \Rightarrow intclA = int A \leq A $ and $clintA \leq clA = A \Rightarrow bdA \leq A$.

On the other hand, it suffices to show $clA \leq A$. If possible let $clA \nleq A \\ \Rightarrow intclA \nleq intA \leq A \Rightarrow intclA \nleq A$\\
And $clintA \leq clA \nleq A$ implies $bdA \nleq A$, a contradiction.
\end{enumerate} 
\end{proof}

\section{Comparative study}
 
In this section, we compare the fuzzy boundary defined here with the different fuzzy boundaries which already exist.

\begin{thm}
For any fuzzy set $A, bdA \leq bd_{II}A$.
\end{thm}

\begin{proof}
For any fuzzy set $A, intclA \leq clA$ and $clint A = cl(clA^c)^c = intclA^c \leq clA^c$\\
So, $intclA \bigwedge clintA \leq clA \bigwedge clA^c$ i.e., $bdA \leq bd_{II}A$.
\end{proof}

\begin{thm}
For any fuzzy set $A, bdA \leq bd_{I}A$.
\end{thm}

\begin{proof}
If $clA \bigwedge clA^c = 0$ then $bd_{I}A = 0_X$.

If $clA \bigwedge clA^c > 0$ then $bd_{I}A = clA$.

So, $clA \bigwedge clA^c \leq bd_{I}A$ and hence $ bdA \leq bd_{I}A$.
\end{proof}

\begin{thm}
For any fuzzy set $A, bdA \leq bd_{III}A$.
\end{thm}

\begin{proof}
We have, if $clA  - intA > 0_X$ then $bd_{III}A = clA$.

If $clA - intA = 0_X$ then $A$ is clopen, so $bdA=A$. But we always have $bd_{III}A \leq clA$.
Therefore, $bdA \leq bd_{III}A$.
\end{proof}
 
\section{Conclusion}

In this paper, we introduced a new type of fuzzy boundary and studied a few properties of this fuzzy boundary. It is also concluded that, this fuzzy boundary is the smallest among different existing fuzzy boundaries. However, the components of exactness will be more in case of this new fuzzy boundary in fuzzy topological space. \\

{\bf{References}}

\end{document}